\documentclass[12pt,oneside,english]{amsart}
\usepackage{lmodern}
\usepackage[T1]{fontenc}
\usepackage[latin9]{inputenc}
\usepackage{babel}
\usepackage{units}
\usepackage{mathtools}
\usepackage{enumitem}
\usepackage{amsthm}
\usepackage{amssymb}
\usepackage[a4paper]{geometry}
\usepackage[pdfusetitle,
 bookmarks=true,bookmarksnumbered=false,bookmarksopen=false,
 breaklinks=false,pdfborder={0 0 0},pdfborderstyle={},backref=false,colorlinks=false]
 {hyperref}

\makeatletter
\numberwithin{equation}{section}
\numberwithin{figure}{section}
\theoremstyle{plain}
\newtheorem{thm}{\protect\theoremname}[section]
\theoremstyle{definition}
\newtheorem{example}[thm]{\protect\examplename}
\theoremstyle{plain}
\newtheorem{prop}[thm]{\protect\propositionname}
\theoremstyle{definition}
\newtheorem{defn}[thm]{\protect\definitionname}
\theoremstyle{remark}
\newtheorem{rem}[thm]{\protect\remarkname}
\theoremstyle{plain}
\newtheorem{lem}[thm]{\protect\lemmaname}
\theoremstyle{plain}
\newtheorem{cor}[thm]{\protect\corollaryname}

\usepackage{faktor}    
\usepackage{amsmath}
\usepackage{scalerel} 

\makeatother

\providecommand{\corollaryname}{Corollary}
\providecommand{\definitionname}{Definition}
\providecommand{\examplename}{Example}
\providecommand{\lemmaname}{Lemma}
\providecommand{\propositionname}{Proposition}
\providecommand{\remarkname}{Remark}
\providecommand{\theoremname}{Theorem}

\begin{document}
\global\long\def\CC{\mathbb{C}}%
\global\long\def\norm#1{\left\Vert #1\right\Vert }%
\global\long\def\sq{\subseteq}%
\global\long\def\HHH{\mathcal{H}}%
\global\long\def\st{\;|\;}%
\global\long\def\inj{\hookrightarrow}%
\global\long\def\ph{\varphi}%
\global\long\def\RR{\mathbb{R}}%
\global\long\def\NN{\mathbb{N}}%
\global\long\def\lim#1{\underset{#1}{lim}}%
\global\long\def\ZZ{\mathbb{Z}}%
\global\long\def\FF{\mathbb{F}}%
\global\long\def\frakm{\mathfrak{m}}%
\global\long\def\lto{\longrightarrow}%
\global\long\def\conv#1{\underset{#1}{\lto}}%
\global\long\def\mf#1{\mathfrak{#1}}%
\global\long\def\cL{\{\}}%
\global\long\def\ep{\varepsilon}%
\global\long\def\cM{\{\}}%
\global\long\def\cU{\mathcal{U}}%
\global\long\def\KK{\mathbb{K}}%

\title[Linear sofic representations of amenable algebras]{Linear sofic representations\\
of amenable algebras}
\author{Benjamin Bachner}
\begin{abstract}
We study the notion of linear sofic representations for algebras,
analogous to the concept of sofic representations for groups. We prove
that for a finitely generated amenable $K$-algebra with no zero divisors,
all linear sofic representations are conjugate. This provides an algebraic
analogue to Elek and Szab\'o's theorem for amenable groups. The proof
relies on a "linear monotiling" technique, constructed using a
theorem by Bre\v{s}ar, Meshulam and \v{S}emrl on locally linearly dependent
operators. Finally, we apply this uniqueness result to the problem
of weak stability in the rank metric, showing that the group algebra
of an amenable group is weakly stable if and only if the group is
residually finite.
\end{abstract}

\maketitle

\section{Introduction}

Elek and Szab\'o have proved \cite{MR2823074} that for a finitely generated
sofic group $\Gamma$, the group $\Gamma$ is amenable if and only
if all sofic representations\footnote{Sometimes referred to in the literature as sofic approximations. Throughout the paper we use the terminology of (linear) sofic representations.} of $\Gamma$ are conjugate. This rigidity
result establishes a deep connection between the asymptotic properties
of finite approximations and the amenability of the group. In this
paper, we investigate the linear algebraic analogue of this phenomenon
for associative algebras. The linear analogue for soficity was introduced
by Arzhantseva and Paunescu \cite{AP_LinearSofic}, using the normalized
rank metric $d\left(A,B\right)=rk\left(A-B\right)=\frac{1}{n}rank\left(A-B\right)$
in place of the normalized Hamming metric. Amenability for algebras
is defined by the existence of a F{\o}lner sequence of subspaces (see
Section \ref{subsec:Amenable}). 

We study the notion of a linear sofic representation for algebras;
An asymptotic homomorphism of a $K$-algebra $A$ is a sequence of
linear maps $\phi_{k}:A\to M_{n_{k}}\left(K\right)$ satisfying 
\[
rk\left(\phi_{k}\left(ab\right)-\phi_{k}\left(a\right)\phi_{k}\left(b\right)\right)\conv{k\to\infty}0\quad\forall a,b\in A.
\]
We call an asymptotic homomorphism $\left(\phi_{k}\right)$ a \emph{linear
sofic representation} if it maximally separates elements, so that $rk\left(\phi_{k}\left(a\right)\right)\conv{k\to\infty}1$
for any nonzero $a\in A$ (our definition is actually stronger than
this; See Section \ref{sec:LSA}). We prove the following analog
of the ``only if'' part of Elek and Szab\'o's theorem. 
\begin{thm}
\label{thm:main}Let $K$ be any field and $A$ be a finitely generated
amenable $K$-algebra without zero divisors. Then there is a unique
conjugacy class of linear sofic representations of $A$.
\end{thm}

We develop a method of ``linear tilings'' for linear sofic representations,
which is analogous to the classical technique of Ornstein and Weiss
\cite{MR910005} of quasi-tiling amenable groups, the latter of which
was implemented by Elek \cite{MR2258227} for sofic representations
of amenable groups. The linear version of tiling relies on a theorem
by Bre\v{s}ar, Meshulam and \v{S}emrl in the study of locally linearly dependent
operators (Section \ref{subsec:LLD}). This linear version turns out
to be much more efficient than the classical tilings for groups; See
Remark \ref{rem:efficient}.

While Elek and Szab\'o established an equivalence for groups, the converse
in the linear setting remains open. We do not know whether an algebra
that admits a unique conjugacy class of linear sofic representations
is necessarily amenable.

In Section \ref{sec:Weak} we apply Theorem \ref{thm:main} to the
study of \emph{weak stability}; Analogous to groups, this property
concerns whether linear sofic representations are asymptotically close
to true representations. 
\begin{thm}[Theorem \ref{thm:weak stability} $($\ref{enu:weak2}$)$]
Let $\Gamma$ be a torsion free amenable group and $K$ a field such
that $K\Gamma$ has no zero divisors. Then $K\Gamma$ is weakly stable
if and only if $\Gamma$ is residually finite.
\end{thm}

A similar property is proven for general amenable algebras that have
representations that ``maximally separate'' its elements in the
rank metric (Theorem \ref{thm:weak stability} (\ref{enu:weak1})),
extending beyond group algebras. 

We conclude with an example in \ref{subsec:abels} of an algebra which
is weakly stable but not stable by taking the group algebra of Abels'
group $A_{p}$. This group was previously shown to be weakly stable
\cite{AP_PStability} but not stable \cite{MR3999445} in permutations.

\subsection*{Acknowledgements}

I would like to thank my advisor, Alex Lubotzky, for his guidance
and helpful discussions. This work is part of the author's Ph.D. thesis
at the Weizmann Institute of Science and was supported by the European
Research Council (ERC) under the European Union\textquoteright s Horizon
2020 (N. 882751). 

\section{Preliminaries}

Throughout the paper, we fix a field $K$ and a finitely generated
$K$-algebra $A$ with generating set $1\in S$ and no zero divisors.
We let $S^{r}$ be the set of all monomials of length at most $r$
and denote by $A_{\le r}=span\left(S^{r}\right)$. 

\subsection{\label{subsec:Amenable}Amenable algebras}

Recall that a discrete group $\Gamma$ is amenable if and only if
there exists a F{\o}lner sequence that exhausts the group, meaning a
sequence $F_{1}\sq F_{2}\sq\ldots$, $\bigcup_{n=1}^{\infty}F_{n}=\Gamma$,
and such that for any $g\in\Gamma$, 
\[
\frac{\left|gF_{n}\cup F_{n}\right|}{\left|F_{n}\right|}\conv{n\to\infty}1.
\]
Analogously, a finitely generated $K$-algebra $A$ is amenable\emph{
}if there exists a sequence of finite dimensional linear subspaces
$W_{1}\sq W_{2}\sq\ldots$ $,\bigcup_{n=1}^{\infty}W_{n}=A$, such
that for any finite dimensional subspace $V\le A$,
\[
\frac{dim_{K}VW_{n}}{dim_{K}W_{n}}\conv{n\to\infty}1.
\]

The amenability of finitely generated algebras was introduced and studied
by Gromov \cite{MR1742309} and Elek \cite{MR1981416}. Elek has proved
characterizations for amenability of finitely generated algebras with
no zero divisors analogous to similar characterizations for groups;
Such an algebra $A$ is amenable if and only if it has no paradoxical
decomposition, if and only if there exists a finitely additive invariant
dimension-measure on $A$ (see \cite{MR1981416} for the relevant
definitions). Bartholdi \cite{MR2448055} has proved that for any
field $K$ and group $\Gamma$, the group algebra $K\Gamma$ is amenable
if and only if $\Gamma$ is amenable. 
\begin{example}
The following classes are amenable.
\begin{itemize}
\item Commutative algebras.
\item Algebras of sub-exponential growth; If 
\[
\gamma_{S}\left(n\right)=dim\left(A_{\le r}\right)
\]
 is the growth function, then $\left(\gamma_{S}\left(n\right)\right)^{\nicefrac{1}{n}}\to1$
(this property is independent of $S$). This includes algebras with
polynomial growth, such as universal enveloping algebras of finite
dimensional Lie algebras (due to the PBW theorem).
\item Group algebras of amenable groups. 
\end{itemize}
\end{example}

We will say that a finite dimensional subspace $W\le A$ is \emph{$\left(V,\varepsilon\right)$-invariant}
if 
\[
dim_{K}VW<\left(1+\varepsilon\right)dim_{K}W.
\]
 For an amenable algebra, for any finite dimensional $V\le A$ and
$\varepsilon>0$ there is a finite dimensional $W\le A$ that is $\left(V,\varepsilon\right)$-invariant.

\subsection{\label{subsec:LS}Linear sofic algebras}

Arzhantseva and P\u{a}unescu \cite{AP_LinearSofic} have introduced
the notion of linear sofic groups and linear sofic algebras as a generalization
of sofic groups, by replacing the finite permutation groups $S_{n}$
equipped with the normalized Hamming metric $d_{H}$ with the groups
$GL_{n}\left(K\right)$ equipped with the normalized rank metric $rk\left(A-B\right)=\frac{1}{n}rank\left(A-B\right)$. 

Fix a nonprincipal ultrafilter $\omega$ and a sequence $n_{k}$ of
natural numbers. The direct product $\prod_{k}M_{n_{k}}\left(K\right)$
is an algebra, and there is a well defined rank function 
\[
rk\left(\left(x_{n}\right)\right)\vcentcolon=\lim{n\to\omega}\,rk\left(x_{n}\right).
\]
The set 
\[
\mathcal{N}_{\omega}=\left\{ \left(x_{n}\right)\in\prod_{k}M_{n_{k}}\left(K\right)\st rk\left(\left(x_{n}\right)\right)=0\right\} 
\]
is a two sided ideal (in fact, a maximal ideal). The quotient is denoted
by 
\[
\prod_{k\to\omega}M_{n_{k}}\left(K\right)\vcentcolon=\prod_{k}M_{n_{k}}\left(K\right)/\mathcal{N}_{\omega},
\]
and $rk$ is a well defined function on this quotient, which is also
known as the \emph{metric ultraproduct} (since $rk\left(x-y\right)$
is a metric). For a more detailed treatment, see \cite[Section 2]{AP_LinearSofic}.

A group $\Gamma$ is $K$-linear sofic if it admits an injective homomorphism
into $\prod_{k\to\omega}GL_{n_{k}}\left(K\right)$ (which is the group
of invertible elements in $\prod_{k\to\omega}M_{n_{k}}\left(K\right)$).
Sofic groups are $K$-linear sofic over any field $K$.

Similarly, a $K$-algebra $A$ is linear sofic if it admits an injective
homomorphism into $\prod_{k\to\omega}M_{n_{k}}\left(K\right)$. The
group $\Gamma$ is $K$-linear sofic if and only if the group algebra
$K\Gamma$ is linear sofic\footnote{It was proven in \cite{AP_LinearSofic} for $K=\CC$ but the proof
works for arbitrary fields.}. 

An injective morphism $\left(\phi_{k}\right):A\to\prod_{k\to\omega}M_{n_{k}}\left(K\right)$
 must satisfy the property that for any nonzero $a\in A$, $\lim{k\to\omega}\;rk\left(\phi_{k}\left(a\right)\right)>0$.
If $ab=0$ are zero divisors, then $\lim{k\to\omega}\;rk\left(\phi_{k}\left(a\right)\right)\le\frac{1}{2}$
or $\lim{k\to\omega}\;rk\left(\phi_{k}\left(b\right)\right)\le\frac{1}{2}$,
so the limit cannot be $1$ for all elements of $A$. In Section \ref{sec:LSA}
we define the notion of a linear sofic representation, where any nonzero
$a\in A$ must satisfy $\lim{k\to\omega}\;rk\left(\phi_{k}\left(a\right)\right)=1$,
so that $\left(\phi_{k}\right)$ is ``maximally separating'', similar
to sofic representations of groups. For this reason, we restrict our
attention to algebras without zero divisors. It might be possible
to generalize this notion to arbitrary linear sofic algebras by defining
a suitable rank function $R:A\to\left[0,1\right]$ such that $R\left(a\right)>0$
for $a\neq0$, and demand that $\lim{k\to\omega}\;rk\left(\phi_{k}\left(a\right)\right)=R\left(a\right)$.

We will say that two sequences of morphisms $\phi_{k},\psi_{k}:A\to M_{n_{k}}\left(K\right)$
are \emph{conjugate} if there is a $M\in\prod_{k\to\omega}GL_{n_{k}}\left(K\right)$
such that $\left(\phi_{k}\right)=M\left(\psi_{k}\right)M^{-1}$ in
$\prod_{k\to\omega}M_{n_{k}}\left(K\right)$. Notice that this is
equivalent to the following: For any $\varepsilon>0$ there is a $B\in\omega$
such that for any $k\in B$, there is some $M_{k}\in GL_{n_{k}}\left(K\right)$
such that $rk\left(\phi_{k}\left(s\right)-M_{k}\psi_{k}\left(s\right)M_{k}^{-1}\right)<\varepsilon$,
$s\in S$. It is enough to prove that for any $\varepsilon>0$, such
an $M_{k}$ exists for all but finitely many $k$.

\subsection{\label{subsec:LLD}Locally linearly dependent operators}

Let $U,V$ be vector spaces over a field $K$. A set of linear operators
$\left\{ T_{1},\ldots,T_{d}\right\} \sq Hom\left(U,V\right)$ is said
to be \emph{locally linearly dependent} if for any $v\in U$, the
set $\left\{ T_{1}v,\ldots,T_{d}v\right\} \sq V$ is linearly dependent. 

Amitsur \cite{MR172902} proved that for any locally linearly dependent
set of size $d$, $\left\{ T_{1},\ldots,T_{d}\right\} $, there is
a nontrivial linear combination $S=\sum_{i=1}^{d}\alpha_{i}T_{i}$
such that 
\[
rank\left(S\right)\le{d+1 \choose 2}-1.
\]
The following stronger result was proved by Aupetit for $K=\CC$ \cite[p. 87]{MR1083349}
and extended by Bre\v{s}ar and \v{S}emrl \cite{MR1621729} for arbitrary infinite
fields. Meshulam and \v{S}emrl \cite{MR1897909} then extended this result
to all fields.
\begin{thm}[Bre\v{s}ar, Meshulam, \v{S}emrl]
\label{thm:LLD} Let $U,V$ be vector spaces over a field $K$ and
$\left\{ T_{1},\ldots,T_{d}\right\} \sq Hom\left(U,V\right)$ locally
linearly dependent operators. Then there exist scalars $\alpha_{1},\ldots,\alpha_{d}\in K$,
not all zero, such that $S=\alpha_{1}T_{1}+\ldots+\alpha_{d}T_{d}$
has $rank\left(S\right)\le d-1$.
\end{thm}

Since we will deal with asymptotic quantities, Amitsur's theorem would
suffice to prove our main theorem. However, we use the stronger
result of Bre\v{s}ar, Meshulam and \v{S}emrl.

\section{\label{sec:LSA}Linear sofic representations}

Recall that for a group $\Gamma$ , a \emph{sofic representation} is
a sequence of maps $\ph_{k}\mapsto S_{n_{k}}$ satisfying
\begin{itemize}
\item $\lim k\;d^{Hamm}\left(\ph_{k}\left(g\right)\ph_{k}\left(h\right),\ph_{k}\left(gh\right)\right)\to0$
for any $g,h\in\Gamma$.
\item $\lim k\;d^{Hamm}\left(\ph_{k}\left(g\right),1\right)\to1$ for any
$g\neq1$.
\end{itemize}
Here, $d^{Hamm}\left(\sigma,\tau\right)=\frac{1}{n}\left|\left\{ j\in\left[n\right]\st\sigma\left(j\right)\neq\tau\left(j\right)\right\} \right|$
is the normalized Hamming metric on the finite symmetric group $S_{n}$.

We define an analog for algebras: A sequence $\phi_{k}:A\to M_{n_{k}}\left(K\right)$
of linear maps is an \emph{asymptotic homomorphism} if 
\[
rk\left(\phi_{k}\left(ab\right)-\phi_{k}\left(a\right)\phi_{k}\left(b\right)\right)\conv{k\to\infty}0
\]
 for any $a,b\in A$.

Fix a nonprincipal ultrafilter $\omega$ on $\NN$. Denote by $K^{\omega}=\prod_{\omega}K$
the ultrapower field, consisting of equivalence classes in $\prod_{n}K$
under the relation $\left(x_{n}\right)_{n}\sim\left(y_{n}\right)_{n}$
if $\left\{ n\st x_{n}=y_{n}\right\} \in\omega$. Note that $K^{\omega}$
embeds in the center of $\prod_{k\to\omega}M_{n_{k}}\left(K\right)$.

Given a sequence of linear maps $\phi_{k}:A\to M_{n_{k}}\left(K\right)$,
there is a natural $K^{\omega}$-linear map $\hat{\phi}:A_{K^{\omega}}\to\prod_{k\to\omega}M_{n_{k}}\left(K\right)$,
where $A_{K^{\omega}}=K^{\omega}\otimes_{K}A$, induced by the bilinear
mapping
\[
\left(\left(x_{k}\right),a\right)\mapsto\left(x_{k}\phi_{k}\left(a\right)\right)_{k}.
\]
We will refer to $\hat{\phi}$ as the mapping induced by $\left(\phi_{k}\right)$. 
\begin{prop}
If $\phi_{k}:A\to M_{n_{k}}\left(K\right)$ is an asymptotic homomorphism,
then the induced map $\hat{\phi}:A_{K^{\omega}}\to\prod_{k\to\omega}M_{n_{k}}\left(K\right)$
is a homomorphism.
\end{prop}

\begin{proof}
Let $x,y\in A_{K^{\omega}}$, so that $x,y$ can each be written as
a finite sum of elements of the form $\left(\alpha_{k}\right)_{i}\otimes a_{i}$.
In order for $\hat{\phi}$ to be a homomorphism it is enough to check
multiplicativity on monomials: 
\[
\hat{\phi}\left(\left(\alpha_{k}\right)_{i}\otimes a_{i}\cdot\left(\alpha'_{k}\right)_{i}\otimes a'_{i}\right)=\hat{\phi}\left(\left(\alpha_{k}\alpha'_{k}\right)\otimes a_{i}a'_{i}\right)=\left(\alpha_{k}\alpha'_{k}\phi_{k}\left(a_{i}a'_{i}\right)\right)_{k}
\]
\[
\hat{\phi}\left(\left(\alpha_{k}\right)_{i}\otimes a_{i}\right)\hat{\phi}\left(\left(\alpha'_{k}\right)_{i}\otimes a'_{i}\right)=\left(\alpha_{k}\alpha'_{k}\phi_{k}\left(a_{i}\right)\phi_{k}\left(a'_{i}\right)\right)_{k}.
\]
Thus, their difference is an element of $\mathcal{N}_{\omega}$ and
hence vanishes in $\prod_{k\to\omega}M_{n_{k}}\left(K\right)$.
\end{proof}
\begin{defn}
Let $d\in\NN$. A linear map $\phi:A\to M_{n}\left(K\right)$ is a
\emph{$d$-approximation} of $\left(A,S\right)$ if there is a subspace
$U\le K^{n}$ such that
\begin{itemize}
\item $\phi$ is \emph{$d$-multiplicative }on $U$; For any $x,y\in A_{\le d}$
and $u\in U$, $\phi\left(xy\right)u=\phi\left(x\right)\phi\left(y\right)u$.
\item $dim\left(U\right)\ge\left(1-\frac{1}{d}\right)n$.
\item For any nonzero $a\in A_{\le d}$, $rk\left(\phi\left(a\right)\right)\ge1-\frac{1}{d}$.
\end{itemize}
\end{defn}

\begin{prop}
\label{prop:Eq_LSA}Let $\phi_{k}:A\to M_{n_{k}}\left(K\right)$ be
a sequence of linear maps. The following are equivalent:
\begin{enumerate}
\item \label{enu:Sep-1}$\left(\phi_{k}\right)$ is an asymptotic homomorphism,
and for any increasing subsequence $\psi_{m}=\phi_{k_{m}},$ $rk\left(\hat{\psi}\left(x\right)\right)=1$
for all $0\neq x\in A_{K^{\omega}}$.
\item \label{enu:LSA-1}For any $d\in\NN$, $\phi_{k}$ is a $d$-approximation
of $\left(A,S\right)$ for large enough $k$.
\end{enumerate}
\end{prop}

If any of the conditions of Proposition \ref{prop:Eq_LSA} are satisfied,
we will say that $\left(\phi_{k}\right)$ is a \emph{linear sofic
representation }of $A$. Notice that this is independent of the generating
set $S$.
\begin{proof}
$ $
\begin{description}
\item [{\ref{enu:LSA-1}$\implies$\ref{enu:Sep-1}}] Let $\varepsilon>0$.
Let $x,y\in A$, and take $d$ large enough so that $x,y\in A_{\le d}$
and $\frac{1}{d}<\varepsilon$. Then for any $d$-approximation $\phi_{k}$,
$\phi_{k}\left(xy\right)-\phi_{k}\left(x\right)\phi_{k}\left(y\right)$
vanishes on a subspace of dimension $\ge\left(1-\frac{1}{d}\right)k_{n}$,
so that $rk\left(\phi_{k}\left(xy\right)-\phi_{k}\left(x\right)\phi_{k}\left(y\right)\right)<\varepsilon$.
Since this happens for any $\varepsilon$ for large enough $k$, $\phi_{k}$
is an asymptotic homomorphism. Let $0\neq x\in A_{K^{\omega}}$ and
$\psi_{m}=\phi_{k_{m}}$ be a subsequence. Then $x$ can be written
as a finite sum $x=\sum_{i}\left(\alpha_{m}^{i}\right)_{m}\otimes a_{i}$,
$\left(\alpha_{m}^{i}\right)_{m}\in K^{\omega}$ and $a_{i}\in A$.
Take $d$ large enough so that all the $a_{i}$ are in $A_{\le d}$,
so now
\[
\hat{\psi}\left(x\right)=\left(\sum_{i}\alpha_{m}^{i}\psi_{m}\left(a_{i}\right)\right)_{m}.
\]
Since for large enough $m$ , $rk\left(\sum_{i}\alpha_{m}^{i}\phi_{m}\left(a_{i}\right)\right)\ge1-\varepsilon$
for nonzero $\alpha_{m}^{i}\phi_{m}\left(a_{i}\right)\in A_{\le d}$,
and moreover $\sum_{i}\alpha_{m}^{i}a_{i}\neq0$ for $\omega$-almost
every $n\in\NN$ (as can be seen by embedding $A_{K^{\omega}}$ in
the ultrapower $A^{\omega})$, we conclude that $rk\left(\hat{\psi}\left(x\right)\right)=1$.
\item [{\ref{enu:Sep-1}$\implies$\ref{enu:LSA-1}}] Let $d\in\NN$. For
any $n$, let 
\begin{align*}
U & =\bigcap_{a,b\in A_{\le d}}ker\left(\phi_{k}\left(ab\right)-\phi_{k}\left(a\right)\phi_{k}\left(b\right)\right)\\
 & =\bigcap_{a,b\in S^{d}}ker\left(\phi_{k}\left(ab\right)-\phi_{k}\left(a\right)\phi_{k}\left(b\right)\right)\le K^{n_{k}}.
\end{align*}
For large enough $k$ we have $rk\left(\phi_{k}\left(ab\right)-\phi_{k}\left(a\right)\phi_{k}\left(b\right)\right)\le\frac{1}{d\left|S^{d}\right|^{2}}$
for any $a,b\in S^{d}$. Thus
\[
dim\left(U\right)\ge n_{k}\left(1-\left|S^{d}\right|^{2}\cdot\frac{1}{d\left|S^{d}\right|^{2}}\right)=\left(1-\frac{1}{d}\right)n_{k},
\]
and $\phi_{k}$ is $d$-multiplicative on $U$. Moreover, for large
enough $k$ we also have $rk\left(\phi_{k}\left(a\right)\right)\ge1-\frac{1}{d}$
for any nonzero $a\in A_{\le d}$; Otherwise, there would be a subsequence
$\psi_{m}=\phi_{k_{m}}$, $x_{m}\in A_{\le d}\backslash\left\{ 0\right\} $
and a constant $\nu>0$ such that $rk\left(\phi_{k_{m}}\left(x_{m}\right)\right)\le1-\nu$.
As $x_{m}$ are all contained in a finite dimensional subspace $A_{\le d}$,
we may view $\left(x_{m}\right)$ as an element of $A_{K^{\omega}}$
-- collecting the coefficients of each basis element in $S^{d}$
for each $m$ gives an element of $K^{\omega}$, and the finite sum
gives $x_{m}$ for each $m$. But $\left(x_{m}\right)$ is mapped
to an element in $\prod_{m\to\omega}M_{n_{k_{m}}}\left(K\right)$
with $rk\le1-\nu$, a contradiction.
\end{description}
\end{proof}
\begin{rem}
    For infinite $K$, the extension to $A_{K^\omega}$  is necessary (for finite $K$  it coincides with $A$); If we only required that any nonzero $a\in A$  satisfy $\lim{k}\;rk\left(\phi_{k}\left(a\right)\right)=1$, then we could, for example, take an enumeration $(\alpha_k)\in K$  of a countably infinite subset and let $\phi_k:K[x] \to M_1(K)=K$ be defined by $\phi_k(x)=\alpha_k$ and extend canonically to a homomorphism. As every polynomial has finitely many roots, every nonzero element $p\in K[x]$  will satisfy $\lim{k}\;rk\left(\phi_{k}\left(p\right)\right)=1$, but it is clearly not even a $1$-approximation of $K[x]$.  Indeed, we may easily find two such maps that are nonconjugate. But the element $(\alpha_k)_k \otimes x -1\otimes 1\in A_{K^\omega}$ will be mapped to $0$.
\end{rem}

Arzhantseva and Paunescu \cite{AP_LinearSofic} have proved that amenable
algebras with no zero divisors are linear sofic. In fact, the construction
is a linear sofic representation:
\begin{prop}
\label{prop:exist}If $A$ is amenable, then it admits a linear sofic
representation.
\end{prop}

\begin{proof}
It is enough to prove that $A$ has $d$-approximations for arbitrary
$d$. Let $d\in\NN$, let $W\le A$ be an $\left(A_{\le2d},\frac{1}{d\left|S^{2d}\right|}\right)$-invariant
subspace and let $p:A_{\le d}\cdot W\to W\cap A_{\le d}\cdot W$ be
an idempotent. For $a\in A$, denote by $m_{a}:A\to A$ the linear
map $x\mapsto ax$. Notice that the zero divisor assumption implies
that $m_{a}$ is injective for $a\neq0$. For any $a\in A$, define
the map $\phi\left(a\right)\in End_{K}\left(W\right)$ by $\phi\left(a\right)w=p\circ m_{a}.$
Let 
\[
U=\left\{ v\in W\st S^{2d}v\in W\right\} \le W.
\]
Then for any $a,b\in S^{d}$ and $u\in U$, 
\[
\phi\left(a\right)\phi\left(b\right)u=papbu=abu=\phi\left(ab\right)u,
\]
and this extends to the linear span. Thus, $rk\left(\phi\left(a\right)\phi\left(b\right)-\phi\left(ab\right)\right)\le1-\frac{dim\left(U\right)}{dim\left(W\right)}$
for any $a,b\in A_{\le d}$. 

Moreover, for any nonzero $a\in A_{\le d}$,
\[
rk\left(\phi\left(a\right)\right)\ge rk\left(p\circ m_{a}\mid_{U}\right)=\frac{dim\left(U\right)}{dim\left(W\right)}.
\]
We are therefore left with proving that $\frac{dim\left(U\right)}{dim\left(W\right)}\ge1-\frac{1}{d}$:
For any $x\in S^{2d}$,
\begin{align*}
dim\left(m_{x}^{-1}\left(W\right)\cap W\right) & =dim\left(W\cap xW\right)=2dim\left(W\right)-dim\left(W+xW\right)\\
 & \ge2dim\left(W\right)-dim\left(A_{\le2d}W\right)\ge\left(1-\frac{1}{d\left|S^{2d}\right|}\right)dim\left(W\right).
\end{align*}
Hence,
\begin{align*}
dim\left(U\right) & =dim\left(\bigcap_{x\in S^{2d}}m_{x}^{-1}\left(W\right)\right)\ge\left(1-\frac{1}{d}\right)dim\left(W\right)\\
 & \implies\frac{dim\left(U\right)}{dim\left(W\right)}\ge\left(1-\frac{1}{d}\right).
\end{align*}
\end{proof}
\begin{rem}
\label{rem:sofic}Notice that in the case that $A=K\Gamma$ where
$\Gamma$ is an amenable group, we may choose the F{\o}lner set $W$
to be the linear span of a suitably invariant F{\o}lner set of $\Gamma$.
Then this construction, when restricted to $\Gamma$, can be slightly
perturbed to true permutations, and hence will be a sofic representation.
As all sofic representations are conjugate, it follows that the linear
extension of any sofic representation $\ph_{n}:\Gamma\to S_{k_{n}}$
is a linear sofic representation of $K\Gamma$. Theorem \ref{thm:main}
will imply that this is, up to small rank perturbation and conjugation,
the only way to get a linear sofic representation of $K\Gamma$.
\end{rem}

Unlike the case of sofic groups, all of which have sofic representations,
we do not know whether (non amenable) linear sofic algebras with no
zero divisors have a linear sofic representation. The ``amplification
trick'' (cf. \cite[Section 5]{AP_LinearSofic}) is not as well-behaved
as it is for sofic representations. 

\section{linear monotilings}

Given a linear map $\phi:A\to M_{n}\left(K\right)$ and a linear subspace
$W\le A$, we will say that a vector $v\in K^{n}$ is a \emph{$W$-root
vector}\footnote{The terminology is inspired by Elek and Grabowski's ``root vectors''
for ballspaces \cite[Section 4]{ElekGrabowski}} for $\phi$ if the mapping 
\[
W\mapsto\phi\left(W\right)v
\]
is injective and multiplicative.

Theorem \ref{thm:LLD} can be applied in order to find many $W$-root
vectors for a finite dimensional $W\le A$ in a linear sofic representation.
\begin{lem}
\label{lem:root}Let $W\le span\left(S^{d}\right)$ be a linear subspace
and let $\phi:A\to M_{n}\left(K\right)$ be a $d$-approximation.
Let $E\le K^{n}$ be a subspace such that 
\[
dim\left(E\right)\le\left(1-\frac{2}{d}\right)n-dim\left(W\right).
\]
Then there exists a $W$-root vector $u\in K^{n}$ such that $\phi\left(W\right)u\cap E=0$.
\end{lem}

\begin{proof}
Let $U\le K^{n}$ be a codimension-$\frac{n}{d}$ subspace on which
$\phi$ is $d$-multiplicative. The restriction $\phi\mid_{U}$ satisfies
for any $x\in W$,
\[
rank\left(\phi\left(x\right)\mid_{U}\right)\ge\left(1-\frac{2}{d}\right)n.
\]
Let $\pi:K^{n}\to K^{n}/E$ be the natural projection. Then for any
$x\in W$, 
\[
rank\left(\pi\circ\phi\left(x\right)\mid_{U}\right)\ge rank\left(\phi\left(x\right)\mid_{U}\right)-dim\left(E\right)\ge dim\left(W\right).
\]
Picking a basis $w_{1},\ldots,w_{dim\left(W\right)}$ for $W$, by
Theorem \ref{thm:LLD} there is a vector $u\in U$ such that $\pi\circ\phi_{U}\left(w_{i}\right)u$
are linearly independent - or equivalently, $\phi\left(w_{i}\right)\mid_{U}u$
are linearly independent modulo $E$. Thus, $u$ is a $W$-root vector
for $\phi$, and $\phi_{U}\left(W\right)u\cap E=0$.
\end{proof}
In \cite{MR2258227}, Elek proved an analogous decomposition for sofic
representations using the Ornstein-Weiss quasi-tiling technique, which
uses finitely many tiles in $\Gamma$ in order to cover most of the
vertices of the graph. Weiss \cite{MR1819193} has proved that residually
finite amenable groups are monotileable, i.e. there is a F{\o}lner sequence
$T_{n}$ and sets $C_{n}\sq\Gamma$ such that for each $n$, $\left\{ T_{n}c\st c\in C_{n}\right\} $
forms a partition of $\Gamma$. It remains an open problem whether
all amenable groups are monotileable. The following can be thought
of as a ``linear monotiling'' lemma for linear sofic representations:
\begin{lem}
\label{lem:monotiling}Let $\phi:A\to M_{n}\left(K\right)$ be a $d$-approximation,
and let $W\le A_{\le d}$ be a nontrivial linear subspace. 

Then there exist $W$-root vectors $v_{1},\ldots,v_{\ell}$ such that
the subspaces $\phi\left(W\right)v_{i}$ are independent, and 
\[
dim\left(\bigoplus_{i=1}^{\ell}\phi\left(W\right)v_{i}\right)\ge\left(1-\frac{2}{d}\right)n-dim\left(W\right).
\]
\end{lem}

\begin{proof}
We find by induction $W$-root vectors $v_{1},\ldots,v_{r}$ such
that $\left\{ \phi\left(W\right)v_{i}\right\} _{i=1}^{r}$ are independent,
$r\le\ell=\frac{\left(1-\frac{2}{d}\right)n}{dim\left(W\right)}-1$.
For $r=1$, if $dim\left(W\right)\ge\left(1-\frac{2}{d}\right)n$
then the statement is trivial with respect to an empty collection
of vectors. Otherwise, by Lemma \ref{lem:root} there exists a $W$-root
vector $v_{1}$.

Assume now that we have found $W$-root vectors $v_{1},\ldots,v_{r}$
such that $\phi\left(W\right)v_{i}$ are independent, and denote $E=\bigoplus_{i=1}^{r}\phi\left(W\right)v_{i}$.
If $dim\left(E\right)\ge\left(1-\frac{2}{d}\right)n-dim\left(W\right)$
then we are done by setting $\ell=r$. Otherwise, by Lemma \ref{lem:root}
we can find a $W$-root vector $v_{r+1}\in K^{n}$ such that $\phi\left(W\right)v_{r+1}\cap E=0$.
Thus, $\left\{ \phi\left(W\right)v_{i}\right\} _{i=1}^{r+1}$ are
independent.
\end{proof}
If we take $\phi_{n}$ to be the linear sofic representation as in
Proposition \ref{prop:exist}, then Lemma \ref{lem:monotiling} implies
that if $V$ is any subspace, any subspace that is sufficiently invariant
can be $\varepsilon$-tiled by $V$. We record this fact which might
be of independent interest.
\begin{cor}
\label{cor:tiling}Let $d\in\NN$, $\varepsilon>0$ and $V\le A_{\le d}$
be a nontrivial finite dimensional subspace, and let $W\le A$ be
a finite dimensional $\left(A_{\le2d},\frac{\varepsilon}{3\left|S^{2d}\right|}\right)$-invariant
subspace. Then there are vectors $v_{1}\ldots,v_{\ell}$ such that
the subspaces $\left\{ V\cdot v_{i}\right\} _{i=1}^{\ell}$ are independent
and 
\[
\bigoplus_{i=1}^{\ell}V\cdot v_{i}\le W
\]
has codimension at most $\varepsilon dim\left(W\right)$.
\end{cor}

\begin{rem}
\label{rem:efficient}Lemma \ref{lem:monotiling} and Corollary \ref{cor:tiling}
imply that the linear version of tiling is much more flexible than
the classical (quasi)tilings. In particular, every finite dimensional
subspace will tile any linear sofic representation (up to some small
codimension), and therefore every F\{\o\}lner sequence will consist of
monotiles. This is not the case for tilings of groups, where monotiles
are a special case of a F{\o}lner sequence, and their existence is not
even known for nonresidually finite groups (it is known if we allow
finitely many tiles; See \cite{MR3905135}). 

As an example, if we take a disk shaped $D\sq\ZZ^{2}$ and a square
$\Sigma\sq\ZZ^{2}$ which is sufficiently $D$-invariant, we cannot
$\varepsilon$-tile $\Sigma$ by $D$ for arbitrary $\varepsilon$
(a circle packing problem); But we can linearly $\varepsilon$-tile
$span_{K}\left(\Sigma\right)$ by $span_{K}\left(D\right)$.
\end{rem}

\section{\label{sec:main}Proof of the main theorem}

For infinite algebras, a linear sofic representation $\phi_{k}:A\to M_{n_{k}}\left(K\right)$
must have $n_{k}\to\infty$;
\begin{prop}
\label{prop:FD}Suppose $\phi_{k}:A\to M_{n_{k}}\left(K\right)$ is
a linear sofic representation such that $n_{k}$ is bounded. Then $A$
is finite dimensional.
\end{prop}

\begin{proof}
Let $n_{k}\le N$ be a uniform bound, and let $k$ be large enough
so that $\phi_{k}$ is an $N+1$-approximation. Then $\left\{ \phi\left(A_{\le N+1}\right)\right\} $
is a subspace of operators, in which every nonzero element has rank
$\ge\left(1-\frac{1}{N+1}\right)n_{k}$. But $n_{k}<N+1$ so that
each operator has full rank. Let $v\in K^{n_{k}}$ be a nonzero vector,
then the linear map $\phi\left(A_{\le N+1}\right)\to K^{n_{k}}$,$x\mapsto xv$
must be injective since a nonzero element in the kernel will correspond
to a singular matrix. But this implies that $dim\left(A_{\le N+1}\right)\le n_{k}\le N$
which can only happen if $A$ is finite dimensional.
\end{proof}
We can now prove the main theorem.
\begin{proof}[proof of Theorem \ref{thm:main}]
We have established the existence of a linear sofic representation
in Proposition \ref{prop:exist}.

Let $\phi_{k},\psi_{k}:A\to M_{n_{k}}\left(K\right)$ be linear sofic
representations. In order for them to be conjugate, as discussed in
Section \ref{subsec:LS} we need to prove that for every $\varepsilon>0$
and for large enough $k$, we have an invertible $M_{k}\in GL_{n_{k}}\left(K\right)$
such that for any $s\in S$,
\[
rk\left(M_{k}\phi_{k}\left(s\right)M_{k}^{-1}-\psi_{k}\left(s\right)\right)<\varepsilon.
\]
Assume first that $n_{k}$ is bounded. By Proposition \ref{prop:FD}
it must be finite dimensional, and combined with the no zero divisor
assumption it must be a division algebra and hence simple. By Artin-Wedderburn
it has a unique irreducible representation, namely $A$ viewed as
a left module. For a $d$-approximation, $d>max\left(n_{k}\right),dim\left(A\right)$,
it must be an actual representation and every nonzero element $a\in A$
must be mapped to a full rank matrix. Thus, if $\phi_{k},\psi_{k}$
are $d$-approximations, they must be the direct sum of copies of
$A$, and therefore they are truly conjugate.

Assume now that $n_{k}\conv{k\to\omega}\infty$. Let $W\le A$ be
a finite dimensional $\left(S,\frac{\varepsilon}{4}\right)$-invariant
subspace, and let $d$ be such that $W\sq A_{\le d}$ and $\frac{1}{d}<\varepsilon$.
Take a $k$ large enough so that $\phi_{k},\psi_{k}$ are $4d$-approximations.
By Lemma \ref{lem:monotiling}, there are $W$-root vectors $v_{1},\ldots,v_{\ell}\in K^{n_{k}}$
for $\phi_{k}$ and $W$-root vectors $u_{1},\ldots,u_{\ell}\in K^{n_{k}}$
for $\psi_{k}$ such that $\phi_{k}\left(W\right)v_{i}$ are independent,
$\psi_{k}\left(W\right)u_{i}$ are independent, and 
\[
dim\left(\bigoplus_{i=1}^{\ell}\phi_{k}\left(W\right)v_{i}\right)=dim\left(\bigoplus_{i=1}^{\ell}\psi_{k}\left(W\right)u_{i}\right)\ge\left(1-\frac{1}{2d}\right)n_{k}-dim\left(W\right).
\]
Notice that $\ell$ is the same for $\left(\phi_{k}\right)$ and $\left(\psi_{k}\right)$
since in Lemma \ref{lem:monotiling} we only assume that the map is
a $d$-approximation.

Take $k$ large enough so that $\frac{1}{4d}n_{k}\ge dim\left(W\right)$.
Then 
\[
dim\left(\bigoplus_{i=1}^{\ell}\phi_{k}\left(W\right)v_{i}\right)=dim\left(\bigoplus_{i=1}^{\ell}\psi_{k}\left(W\right)u_{i}\right)\ge\left(1-\frac{3}{4d}\right)n_{k}.
\]
Let $w_{1},\ldots,w_{m}$ be a basis for $W$ and let $M_{n_{k}}\in GL_{n_{k}}\left(K\right)$
be any invertible map satisfying
\[
M_{n_{k}}\phi_{k}\left(w_{i}\right)v_{j}=\psi_{k}\left(w_{i}\right)u_{j},\;i\in\left[1,m\right],j\in\left[1,\ell\right].
\]
Notice that these are linearly independent vectors so that such an
$M_{n_{k}}$ exists. Then for any $s\in S$ and any $w\in m_{s}^{-1}\left(W\right)\cap W$,
\begin{align*}
 & M_{n_{k}}\phi_{k}\left(s\right)M_{n_{k}}^{-1}\left(\psi_{k}\left(w\right)u_{j}\right)=M_{n_{k}}\phi_{k}\left(s\right)M_{n_{k}}^{-1}\left(\psi_{k}\left(w\right)u_{j}\right)\\
= & M_{n_{k}}\phi_{k}\left(s\right)\phi_{k}\left(w\right)v_{j}=M_{n_{k}}\phi_{k}\left(sw\right)v_{j}=\psi_{k}\left(sw\right)u_{j}\\
= & \psi_{k}\left(s\right)\psi_{k}\left(w\right)u_{j}.
\end{align*}
Observe that 
\[
dim\left(m_{s}^{-1}\left(W\right)\cap W\right)=dim\left(W\cap sW\right)\ge\left(1-\frac{\varepsilon}{4}\right)dim\left(W\right).
\]
Since $M_{n_{k}}\phi_{k}\left(s\right)M_{n_{k}}^{-1}$ and $\psi_{k}\left(s\right)$
agree on the subspace $\bigoplus_{i=1}^{\ell}\psi_{k}\left(m_{s}^{-1}\left(W\right)\cap W\right)u_{i}$,
it follows that 
\[
rk\left(M_{n_{k}}\phi_{k}\left(s\right)M_{n_{k}}^{-1}-\psi_{k}\left(s\right)\right)\le1-\frac{dim\left(m_{s}^{-1}\left(W\right)\cap W\right)\cdot\ell}{n_{k}}
\]
\[
\le1-\frac{\left(1-\frac{\varepsilon}{4}\right)dim\left(W\right)\cdot\ell}{n_{k}}\le1-\left(1-\frac{\varepsilon}{4}\right)\left(1-\frac{3}{4d}\right)<\varepsilon.
\]
This proves that $\left(\phi_{k}\right)$ and $\left(\psi_{k}\right)$
are conjugate.
\end{proof}
\begin{rem}
The above proof implies a notion of \emph{hyperfiniteness} for linear
sofic representations; Recall that by a definition of Elek \cite{MR2455943},
sequence of graphs $\left\{ G_{n}\right\} _{n=1}^{\infty}$ is \emph{hyperfinite}
if for any $\varepsilon>0$ there is a $k_{\varepsilon}\in\NN$ such
that for any $n$ we can erase $\varepsilon\left|E\left(G_{n}\right)\right|$
edges of $G_{n}$ to obtain a graph $G'_{n}$ whose components are
of size at most $k_{\varepsilon}$. We can analogously define a sequence
$\phi_{k}:A\to M_{n_{k}}\left(K\right)$ to be \emph{hyperfinite}
if for any $\varepsilon>0$ there is a $d_{\varepsilon}\in\NN$ such
that for any $k$ there is an idempotent $p_{k}\in M_{n_{k}}\left(K\right)$
of codimension $\le\varepsilon n_{k}$ such that $\left(p_{k}\phi_{k}p_{k}\right)$
splits as a direct sum acting on subspaces of dimension at most $d_{\varepsilon}$.
We have implicitly proved above that every linear sofic representation
of an amenable algebra is hyperfinite. We do not know whether a linear
sofic representation of a nonamenable algebra (if it exists) can be
hyperfinite.
\end{rem}

\section{\label{sec:Weak}Weak stability}

Arzhantseva and Paunescu have introduced the notion of \emph{weak
stability} in permutations. A group $\Gamma$ is weakly stable in
permutations if for every sofic representation $\left(\ph_{n}\right)$
of $\Gamma$, there are true homomorphisms $\left(\psi_{n}\right)$
that are asymptotically close to $\left(\ph_{n}\right)$ (cf \cite[Section 7]{AP_PStability}).
This is a weaker notion than \emph{stability}, which demands that
every \emph{asymptotic homomorphism} of $\Gamma$ be asymptotically
close to homomorphisms. We refer to \cite{MR4929407}
for a thorough treatment of stability of algebras in the rank metric.

We will use here the convention that a representation is a linear,
multiplicative map that is not necessarily unital. 
\begin{rem}
Allowing non unital representations lets us add some zero representations
in order to reach a desired dimension. If we required the maps to
be unital, we could still prove a similar result, but only if we allowed
slight increases in the target dimension $n_{k}$. In the study of
stability, allowing an increase in dimension is commonly referred
to as \emph{flexible} (weak) stability, and was in fact used as the
definition of stability in the rank metric in \cite{ElekGrabowski}
and \cite{MR4929407}. In any case, in the case of a group
algebra we can add representations induced by the trivial representation
of the group, keeping the unitality of the map.
\end{rem}

We define the notion of weak stability of $K$-algebras analogously
to weak stability in permutations. 
\begin{defn}
The algebra $A$ is \emph{weakly stable} if for every $\varepsilon>0$
there is a $d\in\NN$ such that for any $n$ and any $d$-approximation
$\phi:A\to M_{n}\left(K\right)$, there exists a true representation
$\psi:A\to M_{n}\left(K\right)$ such that $rk\left(\phi\left(s\right)-\psi\left(s\right)\right)<\varepsilon$
for any $s\in S$.
\end{defn}

As is standard in the study of stability, this property is equivalent
to the following lifting problem: For every linear sofic representation
$\phi_{k}:A\to M_{n_{k}}\left(K\right)$, the (injective) homomorphism
$\left(\phi_{k}\right):A\to\prod_{k\to\omega}M_{n_{k}}\left(K\right)$
lifts to a homomorphism to the direct product $\prod_{k}M_{n_{k}}\left(K\right)$. 

It follows from Theorem \ref{thm:main} that for $A$ amenable, similarly
to amenable groups, it is weakly stable if and only if it has a sequence
of \emph{true} representations $\psi_{n}:A\to M_{n_{k}}\left(K\right)$
which maximally separate its elements. We first need the following
proposition:
\begin{prop}
\label{prop:amplification}Suppose $\psi_{j}:A\to M_{m_{j}}\left(K\right)$
is a sequence of representations such that $\left(\psi_{j}\right)$
is a linear sofic representation. Then for any sequence $n_{k}\to\infty$
there is a sequence $\rho_{k}:A\to M_{n_{k}}\left(K\right)$ of representations
such that $\left(\rho_{k}\right)$ is a linear sofic representation.
\end{prop}

The following is very similar to the proof of an analogous proposition
for permutations \cite[Proposition 6.1]{AP_PStability}.
\begin{proof}
We can construct an increasing sequence $\left\{ i_{j}\right\} _{j}$
such that $n_{k}>j\cdot m_{j}$ for any $k>i_{j}$. Now, for each
$i_{j}<k\le i_{j+1}$ let $n_{k}=c_{k}m_{j}+r_{k}$, with $r_{k}<m_{j}$
and $c_{k}\ge j$. We then set 
\[
\rho_{k}\left(a\right)\vcentcolon=\left(\psi_{j}\left(a\right)\otimes I_{c_{k}}\right)\oplus0_{r_{k}}\in M_{n_{k}}\left(K\right).
\]
Let $d\in\NN$ and let $j$ be large enough so that $\psi_{j}$ is
a $2d$-approximation and $j\ge2d$. Then for any nonzero $a\in A_{\le d}$,
\[
rank\left(\rho_{k}\left(a\right)\right)=c_{k}\cdot rank\left(\psi_{j}\left(a\right)\right)\ge c_{k}\cdot\left(1-\frac{1}{2d}\right)m_{j}
\]
\begin{align*}
\implies rk\left(\rho_{k}\left(a\right)\right) & \ge\frac{c_{k}\cdot\left(1-\frac{1}{2d}\right)m_{j}}{c_{k}m_{j}+r_{k}}\ge\frac{c_{k}}{c_{k}+1}\left(1-\frac{1}{2d}\right)\\
 & \ge\left(1-\frac{1}{2d}\right)^{2}>1-\frac{1}{d}.
\end{align*}
Thus, $\rho_{k}$ is a linear sofic representation.
\end{proof}
The analogue for algebras of residual finiteness in groups is the
property of being residually finite dimensional (RFD); Every nonzero
$a\in A$ is mapped nontrivially by some finite dimensional representation.
It is clear that an amenable algebra that is weakly stable must be
RFD. 

In the sofic context, constructing a sofic representation for a residually
finite group given by actual homomorphisms is straightforward. By
contrast, finding a linear sofic representation given by true representations
for an RFD algebra is a difficult task, and we do not know if it is
true in general.

When the algebra is a group algebra, however, we have seen that a
sofic representation is a linear sofic representation of the underlying
group algebra (provided that $K\Gamma$ has no zero divisors). Thus:
\begin{thm}
\label{thm:weak stability}Let $K$ be a field and $A$ be an amenable
$K$-algebra with no zero divisors. 
\begin{enumerate}
\item \label{enu:weak1}$A$ is weakly stable if and only if there is a
sequence $\psi_{n}:A\to M_{n_{k}}\left(K\right)$ of representations
such that $\left(\psi_{n}\right)$ is a linear sofic representation.
\item \label{enu:weak2}If $A=K\Gamma$ is the group algebra of an amenable
group $\Gamma$, then $A$ is weakly stable if and only if $\Gamma$
is residually finite.
\end{enumerate}
\end{thm}

\begin{proof}
Suppose there is a sequence $\left(\psi_{n}\right)$ of representations
and let $\phi_{k}:A\to M_{n_{k}}\left(K\right)$ be a linear sofic
representation. By Proposition \ref{prop:amplification} there is a
sequence $\rho_{k}:A\to M_{n_{k}}\left(K\right)$ of representations
that is a linear sofic representation. By Theorem \ref{thm:main},
$\left(\phi_{k}\right)$ and $\left(\rho_{k}\right)$ are conjugate,
so that $\left(\phi_{k}\right)$ lifts to a representation into $\prod_{k}M_{n_{k}}\left(K\right)$
(some conjugation of $\left(\rho_{k}\right)$). The other direction
is easy, which proves $($\ref{enu:weak1}$)$.

Suppose now that $A=K\Gamma$ where $\Gamma$ is a residually finite
amenable group. Then there is a sequence of homomorphisms of $\Gamma$
that is a sofic representation (cf \cite[Corollary 6.2]{AP_PStability}).
We have seen (Remark \ref{rem:sofic}) that the linear extension of
a sofic representation is a linear sofic representation of $K\Gamma$.
It follows from $($\ref{enu:weak1}$)$ that $K\Gamma$ is weakly
stable. Conversely, if $K\Gamma$ is weakly stable then a linear sofic
representation of $K\Gamma$ will be close to one given by true representations.
Thus, $\Gamma$ is residually linear and hence residually finite by
a famous theorem of Malcev, proving $($\ref{enu:weak2}$)$.
\end{proof}
According to a famous conjecture of Kaplansky, $K\Gamma$ has no zero
divisors whenever $\Gamma$ is a torsion-free group. This conjecture
is still open, but it is known for large classes of groups. For instance,
Kropholler, Linnell and Moody \cite{MR964842} have proved that if
$\Gamma$ is a torsion free elementary amenable group, then $K\Gamma$
embeds in a division ring and therefore has no zero divisors.

\subsection{\label{subsec:abels}An algebra that is weakly stable but not stable}

Recall that Abels' group $A_{p}$ (see \cite{MR564423}) is the linear
group 
\[
\left\{ \begin{pmatrix}1 & * & * & *\\
 & p^{m} & * & *\\
 &  & p^{n} & *\\
 &  &  & 1
\end{pmatrix}\quad\st\quad*\in\ZZ\left[\frac{1}{p}\right],m,n\in\ZZ\right\} \sq GL_{4}\left(\mathbb{Q}\right),
\]
where $p$ is a prime. It is a finitely presented solvable group,
and hence elementary amenable. Since no element in $A_{p}$ has nontrivial
roots of unity as eigenvalues, it is also torsion-free. By the discussion
above, for any field $K$ the group algebra $KA_{p}$ has no zero
divisors. It is therefore weakly stable by Theorem \ref{thm:weak stability}.

Becker, Lubotzky and Thom have proved in \cite[Corollary 8.7]{MR3999445}
that $A_{p}$ is not stable in permutations (although it is weakly
stable in permutations by \cite[Theorem 1.1]{AP_PStability}). Based
on this proof, Elek and Grabowski proved in \cite[Theorem 4]{ElekGrabowski}
that $A_{p}$ is not stable in the rank metric (when we consider maps
into $GL_{n}\left(\CC\right)$). Observe that this precisely means
that the algebra $\CC A_{p}$ is not stable (their proof in fact holds
over any field $K$). It follows that $\CC A_{p}$ is weakly stable,
but not stable.

\bibliographystyle{plain}
\bibliography{bibi}

\end{document}